\documentclass[a4paper,10pt]{article}

\usepackage[english]{babel}
\usepackage{natbib}

\usepackage{ucs}
\usepackage[utf8x]{inputenc} 
\usepackage[T1]{fontenc} 
\usepackage{lmodern}%
\usepackage[super]{nth}

\usepackage{amsmath}
\usepackage{amsfonts}
\usepackage{mathrsfs}
\usepackage{amssymb}
\usepackage{amsthm}

\usepackage{tikz-cd}

\renewenvironment{tikzcd}{\begin{oldtikzcd}}{\end{oldtikzcd}\vspace{-.5cm}}

\newtheorem{theorem}{Theorem}[section]
\newtheorem{lemma}[theorem]{Lemma}
\newtheorem{proposition}[theorem]{Proposition}
\newtheorem{corollary}{Corollary}[theorem]

\theoremstyle{definition}
\newtheorem{definition}[theorem]{Definition}
\newtheorem{remark}[corollary]{Remark}
\newtheorem{notation}[theorem]{Notation}
\newtheorem{example}{Example}[section]

\newcommand{\N}{\mathbb{N}}
\newcommand{\Z}{\mathbb{Z}}
\newcommand{\PP}{\mathbb{P}}

\DeclareMathAlphabet{\mathpzc}{OT1}{pzc}{m}{it}
\newcommand{\cl}{\scalebox{1.2}{$\mathpzc{Cl}$}}
\newcommand{\rproj}{\scalebox{1.18}{$\mathpzc{Proj}$}}

\newcommand{\apl}[2]{#1 \! \rightarrow \! #2}
\newcommand{\overapl}[3]{#1 \! \overset{#3}{\rightarrow} \! #2}
\newcommand{\aplllarga}[3]{\begin{oldtikzcd}[column sep=2.5em, ampersand replacement=\&] 
                           #1 : #2 \rar \& #3 \end{oldtikzcd}}

\DeclareMathOperator{\bl}{bl}
\DeclareMathOperator{\Bl}{Bl}
\DeclareMathOperator{\Cl}{Cl}

\DeclareMathOperator{\s}{Spec}
\DeclareMathOperator{\Pic}{Pic}

\DeclareMathOperator{\Hilb}{Hilb}

\DeclareMathOperator{\Hom}{Hom}
\DeclareMathOperator{\Id}{Id}
\DeclareMathOperator{\im}{em}

\DeclareMathOperator{\Sch}{Sch}
\DeclareMathOperator{\Set}{Set}

\DeclareMathOperator{\adm}{adm}

\title{On the Universal Scheme of \(r\)-relative clusters 
  of a family}
\author{Pau Brustenga i Moncusí}

\begin{document}
\setlength{\abovedisplayskip}{5pt} \setlength{\belowdisplayskip}{5pt}
\setlength{\abovedisplayshortskip}{5pt} \setlength{\belowdisplayshortskip}{5pt}

\maketitle



\section*{Introduction}

Kleiman introduced in \citep{kleiman-multiple-point-1981} the iterated blowups,
varieties \(X_r\)
which naturally parametrize ordered clusters of \(r\)
(possibly infinitely near) points of a variety \(X\),
for each \(r\).
These have proved to be very useful in enumerative geometry, and in other areas
of algebraic geometry, especially in the case when \(X\)
is a surface.  See for instance \citep{kleiman-multiple-point-1981},
\citep{Kleiman-Enumerating-1999}, \citep{roe-conditions-2001},
\citep{roe-existence-2001}, \citep{de-poi-threefolds-2003},
\citep{Kleiman-Node-2004}, \citep{fernandez-de-bobadilla-moduli-2005},
\citep{Ran-Geometry-2005} and \citep{alberich-carraminana-enriques-2005}.\\

Following the general philosophy that Grothendieck expounds in his \'El\'ements
de G\'eometrie Alg\'ebrique'', given a scheme \(B\),
we generalise the concept of point cluster of a variety to the \(B\)-point
cluster of \(B\)-variety
\(\mathcal{S}\),
or relative clusters of a family of varieties \(\apl{\pi\!:\! \mathcal{S}}{B}\).
We define schemes \(\Cl_r\)
that generalise the iterated blow ups in the sense that they naturally
parametrize
clusters of \(r\)
\(B\)-points of a \(B\)-variety
\(\mathcal{S}\),
or \(r\)-relative
clusters of a family of varieties \(\apl{\pi\!:\! \mathcal{S}}{B}\),
and we prove their existence if \(B\) is
projective and \(\mathcal{S}\) quasiprojective (Theorem~\ref{thr:existence of Cl_r}).\\

Kleiman's varieties $X_r$ can be identified as our \(\Cl_r\)
for the particular ``absolute'' case when $B=\s K$.  In the case of surfaces
(especially $\PP^2$) these varieties have received a good deal of attention in
the literature for themselves (i.e., not just for applications) see for instance
\citep{harbourne-iterated-1988}, \citep{paxia-flat-1991} ---who does not consider
the possibility of infinitely near points --- \citep{roe-varieties-2004} or
\citep{kleiman-enriques-2011}.  Our approach, taking the point of view of
universal families and representable functors, is closer in spirit to that of
\citep{harbourne-iterated-1988}.  The relations between Kleiman varieties and
Hilbert schemes became a recurring theme which was clarified in
\citep{kleiman-enriques-2011}; it is natural to hope that the analogous forgetful
maps (eliminating the ordering of the \(B\)-points)
from the universal schemes \(\Cl_r\)
will be useful in the study of the relative Hilbert scheme
\(\Hilb_{\mathcal{S}/B}\).
Harbourne goes even further in the case of $\PP^2$, considering isomorphisms
between fibers of the universal family and the corresponding moduli problem,
which gives rise to a quotient stack.  We do not deal with these issues here.
\\

When \(B\) is regular an explicit construction of \(\Cl_{r+1}\)
is possible, in terms of iterated blow ups centred at suitable sheaves of ideals
with support at a subscheme of \(\Cl_r^2:=\Cl_r\times_{\Cl_{r-1}} \Cl_r\),
which fails to be Cartier only along the diagonal.  To this extent, the
construction is indeed analogous the one in \citep{kleiman-multiple-point-1981}
for the absolute case, which works by iteratively blowing up along diagonals; but
new phenomena appear in the relative case.  To be precise, we show that there is
an open subscheme \(V\subseteq \Cl_{r+1}\),
parametrizing the clusters of sections in which the image of the last section is
not contained in the exceptional divisor of the blow up of \(\Cl_r\), and a stratification
\(\apl{\sqcup_{p\in P} \Cl^p}{\Cl_r^2}\),
by locally closed subschemes, such that \(V\)
is isomorphic to a subset of strata \(\sqcup_{p\in\adm}\Cl^p\)
(Theorems~\ref{theorem of birrationaliti} and~\ref{the end theorem}).  For some
strata whose closure intersects the diagonal, the corresponding component of
\(\Cl_{r+1}\)
is (an open subset of) a blowup as above; since strata are often singular along
the diagonal, the centre of the blowup is not always uniquely determined (but the
blowup map is uniquely determined).  It also happens that infinitely near
\(B\)-points exist,
which are not limits of ``ordinary''  \(B\)-points.\\

In Section~\ref{examples} we show a few simple examples illustrating these new
phenomena, in the case of families of surfaces.


\renewcommand{\abstractname}{Acknowledgements}
\begin{abstract}
  I thank my thesis advisor, Joaquim Roé, for his great support.
\end{abstract}

\section{Preliminaries}\label{preliminaries}
On the whole article we fix an algebraically closed base field \(K\),
we consider all scheme a \(K\)-scheme and all morphism a morphism of \(K\)-schemes.

Let \(Y\)
be a scheme and \(y\in Y\)
a point. We denote \(k(y)\)
the residue field of \(y\)
and \(\apl{\im_y\!:\!\s(k(y))}{Y}\)
the embedding
of \(y\)
in \(Y\).
Given another scheme \(X\),
we denote \(X^y\)
the product \(\s(k(y))\times_{\s(K)} X\).
Given a morphism \(\apl{f\!:\! X}{Y}\)
and a subscheme \(W\subseteq Y\),
we denote \(f^{-1}(W)\)
the schematic pre-image (i.e. the pull back by \(f\)
of the inclusion \(\apl{W}{Y}\),
which is a subscheme of \(X\))
and \(X_y\) the fibre of \(f\) at \(y\) (i.e. \(X_y:=f^{-1}(y)\)).



Consider the following diagram. 
\[
  \begin{tikzcd}
    X  \rar{} \dar{}  & Y \rar{} \dar{}  & Z \dar{}   \\
    X' \rar{}  & Y' \rar{} & Z'  \\
  \end{tikzcd}
\]
We call it a \emph{Cartesian diagram} when both squares are Cartesian.  We use
the following fact without further reference: if the right hand square is
Cartesian then the left hand square is Cartesian if and only if so is the big one
obtained by composing the rows, see \citep[p.89]{hartshorne-algebraic-1977}.

\begin{definition}\label{def:loally-quasiprojective}
  We call a scheme which is a finite or countable disjoint union of
  quasiprojective schemes a \emph{piecewise quasiprojective} scheme.   
\end{definition}

\begin{lemma}\label{lem:imatge section is closed}
  Given a separated morphism \(\apl{\alpha\!:\!X}{Y}\),
  any section \(\apl{a\!:\!Y}{X}\) of \(\alpha\) is a closed embedding.
\end{lemma}

\begin{proof}
  It is straightforward to check that the diagram
  \[
    \begin{tikzcd}
      Y \arrow{r}{a} \dar[swap]{a}  &
      X \arrow{d}{\Id_{X}\times_Y (a\circ \alpha)} \\
      X \arrow[hook,swap]{r}{\Delta} & X \times_Y X \\
    \end{tikzcd}
  \]
  is Cartesian, where \(\Delta\)
  is the diagonal, which is a closed embedding by hypothesis. Therefore so is the
  morphism \(a\).
\end{proof}

\begin{lemma}\label{lem:first}
  Given the following diagram where the square is Cartesian,
  \[
    \begin{tikzcd}
      &  X' \ar{r}{\beta} \dar[swap]{g}  & Y' \ar{d}{f}  \\
      Y \rar[swap]{a}  &  X \rar[swap]{\alpha} & Y \\
    \end{tikzcd}
  \]
  if \(\apl{a \!:\!Y }{X }\)
  is a section of \(\alpha\)
  then there is a unique section \(\apl{b\!:\!Y'}{X'}\)
  of \(\beta\) that makes the following diagram Cartesian.
  \[
    \begin{tikzcd}
      Y' \ar[dashed]{r}{b} \ar{d}[swap]{f}  & X' \ar{d}{g} \\
      Y \ar{r}[swap]{a} & X \\
    \end{tikzcd}
  \]
  
\end{lemma}

\begin{proof}
  Consider the following Cartesian diagram.
  \[
    \begin{tikzcd}
      Z \ar{r}  \ar{d} & X' \ar{d}[swap]{g}   \ar{r}{\beta} & Y' \ar{d}{f} \\
      Y \ar{r}[swap]{a} & X \ar{r}[swap]{\alpha}  & Y \\
    \end{tikzcd}
  \]
  Since the big square of the diagram is Cartesian, and obviously
  \(Y'\cong Y\times_{Y}Y'\), we have \(Z\cong Y'\) and the claim follows.
\end{proof}

We will use the above lemma to construct sections of a given morphism.

\begin{remark}\label{rmk:post_first_lema}
  Under the conditions of Lemma~\ref{lem:first}, if the morphism \(\alpha\)
  is separated then so is \(\beta\)
  and, by Lemma \ref{lem:imatge section is closed}, the morphisms \(a\)
  and \(b\) are closed embeddings.
\end{remark}

\begin{definition}
  Let \(\apl{\pi\!:\!\mathcal{S}}{B}\) be a morphism. A \emph{sections
    family} (or sf) of \(\pi\) is a couple \((Y,\rho)\) consisting of a scheme
  \(Y\) and a morphism \(\apl{\rho\!:\! Y\times B}{\mathcal{S}}\) such that the
  composition \(\pi\circ\rho\) is the projection over \(B\).

  Let \((X,\psi)\)
  be an sf of \(\pi\),
  we call \((X,\psi)\)
  a \emph{Universal sections family} (or Usf) of \(\pi\)
  if it satisfies the following universal property: for any sf \((Y,\rho)\)
  of \(\pi\)
  there is a unique morphism \(\apl{f\!:\!Y}{X}\)
  such that \(\rho=\psi\circ (f\times \Id_{B})\).
\end{definition}
If a Usf of a morphism \(\apl{\pi\!:\! \mathcal{S}}{B}\)
exists, by abstract nonsense it is unique up to unique isomorphism.\\

Given a morphism \(\apl{\pi\!:\!\mathcal{S}}{B}\)
there is the contravariant functor \(\apl{S.\_\!:\!  \Sch}{\Set}\)
corresponding to the parameter space problem of sections of \(\pi\)
defined as follows.  For every scheme \(Y\), let
\[
  S.Y:=\{\rho\in \Hom_{\Sch}(Y\times B,\mathcal{S}) \mbox{ such that } (Y,\rho)
  \mbox{ is an sf of }\pi\}
\]
and for every morphism \(\apl{f\!:\! Y'}{Y}\), let \(S.f:=(f\times
\Id_{B})^*\), see \citep[C.2.]{grothendieck-sb-II}.

\begin{definition}\label{def:associeted section}
  Given \((Y,\rho)\)
  an sf of a morphism \(\apl{\pi\!:\! \mathcal{S}}{B}\)
  and \(y\in Y\), slightly abusing language 
  we call the morphism
  \(\apl{\underline{y}:= S.\im_{y}(\rho) \!:\! B^{y}}{\mathcal{S}}\)
  the \emph{section associated to} \(y\).
\end{definition}

If \(y\)
is a closed point, since \(k(y)\cong K\) the morphism \(\underline{y}\) can be
identified with a section of \(\pi\).

\begin{proposition}\label{def:functor S._}
  A scheme \(X\)
  represents the functor \(S.\_\)
  associated to a morphism \(\apl{\pi \!:\!  \mathcal{S}}{ B}\)
  if and only if \(X\) is the Usf of \(\pi\).
\end{proposition}
\begin{proof}
  We just sketch the proof.  If \((X,\psi)\)
  is the Usf of \(\pi\)
  then we can construct a natural isomorphism
  \(\apl{\mu\!:\! S.\_}{\Hom_{\Sch}(\_,X)}\)
  as follows.  Given a scheme \(Y\),
  for each \(\rho\in S.Y\),
  by the universal property of \((X,\psi)\),
  we define \(\mu_Y(\rho)\in \Hom_{\Sch}(Y,X)\)
  as the unique morphism such that \(\rho=S.(\mu_Y(\rho))(\psi)\).
  
  If \(X\)
  represents \(S.\_\),
  we fix a natural isomorphism \(\apl{\eta\!:\! S.\_}{\Hom_{\Sch}(\_,X)}\).
  If \(\psi:=\eta_X^{-1}(\Id_X)\in S.X\),
  then the couple \((X,\psi)\) is the Usf of \(\pi\).
\end{proof}

\begin{corollary}\label{cor:the Usf has just all sections of the morphism}
  If the Usf \((X,\psi)\)
  of a morphism \(\apl{\pi\!:\! \mathcal{S}}{B}\) exists, the map
  \(\apl{\sigma}{\underline{\sigma}}\) from the closed points of \(X\) to
  the set of sections of \(\pi\) is bijective.
\end{corollary}




\begin{theorem}[Grothendieck]\label{thm:exist-Usf}
  Given a morphism \(\apl{\pi\!:\!\mathcal{S}}{B}\),
  if \(B\)
  is proper and \(\mathcal{S}\)
  is piecewise quasiprojective, then the Usf \((X,\psi)\)
  of \(\pi\)
  exists and the scheme \(X\)
  is piecewise quasiprojective.
\end{theorem}
\begin{proof}
  We consider the scheme \(\mathcal{S}\) as the finite or countable disjoint union of
  quasiprojective schemes \(\bigsqcup_{i\in I}\mathcal{S}_i\). For each \(i\in I\)
  the Usf \((X_i,\psi_i)\) of \(\apl{\pi|_{\mathcal{S}_i}\!:\!\mathcal{S}_i}{B}\)
  exists, see~\citep[4.c]{grothendieck-sb-4-hilbert} and \(X_i\) is piecewise
  quasiprojective. So, \(X:=\bigsqcup_{i\in
    I}X_i\) and the morphism \(\psi:= \bigsqcup_{i\in I} \psi_i\circ \theta\) are
  the Usf of \(\pi\), where \(\theta\) is the isomorphism from  \(X\times B\) to \(\bigsqcup_{i\in I}(X_i\times B)\). 
\end{proof}
For an alternative exposition to~\citep[4.c]{grothendieck-sb-4-hilbert},
see~\citep[\nth{2} ex. after Th.6.6]{nitsure-construction-2005}.\\

The rest of this section is devoted to a general construction that appears on
some proofs. It is not our main object of study.


\begin{definition}\label{def:W-ssf}
  Let \(W\) be a scheme. Let \(\apl{\pi\!:\! \mathcal{S}}{B}\) be a morphism with
  \(\mathcal{S}\) a \(W\)-scheme and \(\apl{\pi'\!:\! \mathcal{S}}{W\times B}\)
  the product morphism of \(\apl{\pi\!:\!\mathcal{S}}{B}\)
  and \(\apl{\mathcal{S}}{W}\).
  Given \(\overapl{Y}{W}{f}\)
  a \(W\)-scheme
  and a sf \(\apl{\rho\!:\! Y\times B}{\mathcal{S}}\)
  of \(\pi\),
  we call the triplet \((Y,\rho,f)\)
  a \emph{\(W\)-split
    section family} (or \(W\)-ssf)
  of \(\pi'\)
  if the morphism \(\rho\)
  is a \(W\)-morphism.
  (Here \(B\)
  is not assumed to be a \(W\)-scheme,
  but \(Y\times B\) is a \(W\)-scheme via the first projection.)

  Let \((X,\alpha,g)\)
  be a \(W\)-ssf
  of \(\pi'\),
  we call \((X,\alpha,g)\)
  a \emph{Universal \(W\)-split
    sections family} (or \(W\)-Ussf)
  of \(\pi'\)
  if it satisfies the following universal property: for any \(W\)-ssf
  \((Y,\rho,f)\)
  of \(\pi'\)
  there is a unique morphism \(\apl{G\!:\! Y}{X}\)
  such that \(\rho=\alpha\circ (G\times \Id_{B})\).
\end{definition}
It is straightforward to check that in this case the condition
\(f=g\circ G\) is satisfied.
If a \(W\)-Ussf
of \(\pi'\) exists, by abstract nonsense it is unique up to unique isomorphism.

\begin{remark}\label{rmk:equivalent-def-of-ssf}
  In the Definition~\ref{def:W-ssf} the morphism \(\rho\)
  is a \(W\)-morphism if and only if \(\pi'\circ\rho=f\times \Id_B\)
\end{remark}

\begin{lemma}\label{lem:triple-uni-property}
  Let \(W\) be a scheme. Let \(\apl{\pi\!:\! \mathcal{S}}{B}\) be a morphism with
  \(\mathcal{S}\) a \(W\)-scheme and \(\apl{\pi'\!:\! \mathcal{S}}{W\times B}\)
  the product morphism of \(\apl{\pi\!:\!\mathcal{S}}{B}\)
  and \(\apl{\mathcal{S}}{W}\).
  If the Usf \((X,\psi)\)
  of \(\pi\)
  and the Usf \((BW,\varphi_W)\)
  of the projection \(\apl{W\times B}{B}\)
  exist, then the \(W\)-Ussf
  of \(\pi'\) exists too.
\end{lemma}
\begin{remark}
  The Usf \((BW,\varphi_W)\) and the scheme of morphisms from \(B\)
  to \(W\) represent isomorphic functors, see \citep[C.2.]{grothendieck-sb-II} or
  \citep[4.c]{grothendieck-sb-4-hilbert}. So its existence is equivalent, and if
  they exist they are isomorphic.
\end{remark}
\begin{proof}
  The couple \((X,\pi'\circ\psi)\)
  is an sf of the projection \(\apl{W\times B}{B}\).
  By the universal property of \((BW,\varphi_W)\)
  there is a unique morphism \(\apl{h\!:\! X}{BW}\)
  such that \(\pi'\circ\psi=\varphi_W\circ(h\times \Id_B)\).

  The couple \((W,\Id_{(W\times B)})\)
  is an sf of \(\apl{W\times B}{W}\).
  By the universal property of \((BW,\varphi_W)\)
  there is a unique morphism \(\apl{i\!:\! W}{BW}\)
  such that \(\Id_{(W\times B)}=\varphi_W\circ(i\times \Id_B)\).
  Given a closed \(t\in B\)
  we define the morphism \(\apl{\lambda_t \!:\! BW }{W }\)
  as the composition of the morphism
  \(\apl{\varphi_W\circ(\Id_{BW}\times \im_t)\!:\!BW\cong BW\times k(t)}{W\times
    B}\) with \(\apl{W\times B}{W}\).  By definition \(\lambda_t\circ i=\Id_W\),
  so \(i\)
  is a closed embedding, because it is a section of \(\lambda_t\)
  which is separated and surjective, and we consider \(W\) as a subscheme of \(BW\).

  Now, if \(Z := h^{-1}(W)\),
  it is straightforward to check that the triplet
  \(\left( Z,\psi|_{(Z\times B)},h|_Z \right)\) is the \(W\)-Ussf of \(\pi'\).
\end{proof}


\begin{corollary}\label{thm:exist-W-Ussf}
  If \(W\)
  and \(\mathcal{S}\)
  are piecewise quasiprojective and \(B\)
  is projective then the \(W\)-Ussf
  \((X,\alpha,g)\)
  of \(\pi'\)
  exists and \(X\)
  is piecewise quasiprojective.
\end{corollary}
\begin{proof}
  The scheme \(W\times B\) is piecewise quasiprojective. So by
  Theorem~\ref{thm:exist-Usf} the hypotheses of 
   Lemma~\ref{lem:triple-uni-property} are satisfied and the claim follows.
\end{proof}
\begin{definition}
  Let \(W\) be a scheme. Let \(\apl{\pi\!:\! \mathcal{S}}{B}\) be a morphism with
  \(\mathcal{S}\) a \(W\)-scheme and \(\apl{\pi'\!:\! \mathcal{S}}{W\times B}\)
  the product morphism of \(\apl{\pi\!:\!\mathcal{S}}{B}\)
  and \(\apl{\mathcal{S}}{W}\).
  Given \((Y,\rho,f)\)
  a \(W\)-ssf
  of \(\pi'\)
  and \(y\in Y\),
  we denote \(\mathcal{S}_{f(y)}\)
  the fibre at \(f(y)\in W\)
  of the morphism \(\apl{\mathcal{S}}{W}\).
  Then the morphism \(\apl{\underline{y} \!:\! B^y }{\mathcal{S}}\)
  has image contained in \(\mathcal{S}_{f(y)}\)
  and we call its corestriction \(\apl{\tilde{y}\!:\! B^y }{\mathcal{S}_{f(y)}}\)
  the \(W\)-\emph{split section associated to} \(y\).
\end{definition}


\section{Clusters of sections}\label{Clusters of sections}

\begin{notation}\label{def:a-ascension}
  Let \(\apl{a\!:\!Y}{X}\)
  be a closed embedding corresponding to a finite type sheaf of ideals (this is
  always the case if \(a\) is a closed embedding and \(X\) is locally Noetherian). We denote \(\apl{\bl_a\!:\!\Bl_a(X)}{X}\)
  the blow up of \(X\) at \(a(Y)\).
\end{notation}

\begin{definition}
  We call a \emph{family} a separated surjective flat morphism \(\apl{\pi\!:\! \mathcal{S}}{B}\)
  of finite type where \(B\) is of finite type and irreducible and the
  generic fibre is integral.
\end{definition}

For this section we fix a family \(\apl{\pi\!:\!\mathcal{S}}{B}\) and an integer \(r\ge 0\).\\

Given a projective variety \(S\),
a \emph{cluster} over \(S\)
is a finite set of points \(\mathcal{K}\)
of \(S\) or of a finite sequence of point centred blow ups of \(S\)
such that, for every \(p\in \mathcal{K}\),
if \(q\)
is a point such that \(p\)
is infinitely near to \(q\),
then \(q \in \mathcal{K}\),
see \citep{casas-alvero-singularities-2000}.  For each \(r\ge 1\)
there is a family of smooth projective varieties \(X_r\)
which parametrizes ordered \(r\)-point clusters of \(S\),
see 
\citep{kleiman-multiple-point-1981}.\\

Let \(\apl{\sigma\!:\! B}{\mathcal{S}}\)
be a section of \(\apl{\pi\!:\! \mathcal{S}}{B}\).
Let \(\tilde{\pi}\) be \(\pi\circ\bl_{\underline{\sigma}}\)
and \(E_\sigma\) be the exceptional divisor in \(\Bl_{\underline{\sigma}}(\mathcal{S})\).

A section of \(\tilde{\pi}\)
whose image is contained in \(E_\sigma\)
is said to belong to the \emph{first infinitesimal neighbourhood} of the section
\(\sigma\).
Inductively, for an integer \(k> 0\)
a section \(\tau\)
belongs to the \emph{\((k+1)\)-th
  infinitesimal neighbourhood} of \(\sigma\)
if \(\tau\)
belongs to the first infinitesimal neighbourhood of a section in the \(k\)-th
infinitesimal neighbourhood of \(\sigma\).
A section \emph{infinitely near} to \(\sigma\)
is a section in some infinitesimal neighbourhood of \(\sigma\).\\

\begin{lemma}\label{lem:flat-and-baseChange-blowup}
  Let \(W\)
  be a ground scheme. Let \(\overapl{X}{W}{f}\)
  be a finitely presented \(W\)-scheme.
  If \(Y\! \overset{i}{\rightarrow} \!X\)
  is a closed subscheme which is flat over
  \(W\),
  then the blow up \(\overline{X}\)
  of \(X\)
  at \(Y\)
  commutes with every base change \(\apl{W'}{W}\).
  Furthermore, if \(X\) is flat over \(W\) then so is \(\overline{X}\)
\end{lemma}
\begin{proof}
  Fix a base change \(\apl{W'}{W}\). Consider
  \(\apl{q\!:\! X'}{X}\) and \(\apl{Y'}{Y}\) the base change of \(X\), \(Y\) and
  \(\apl{i'\!:\!Y'}{X'}\)  the pull back of \(i\). We call \(\overline{X'}\) the
  blow up of \(X'\) at \(Y'\).
  
  Consider \(\mathscr{I}\) the defining ideal sheaf of \(Y\) in \(X\) and
  \(\mathscr{I}'\) the defining ideal sheaf of \(Y'\) in \(X'\) (i.e. the kernel of
  \(\apl{\mathscr{O}_{X'}}{(i')_{*}\mathscr{O}_{Y'}}\)).
  There is a natural isomorphism (see \citep
  {vakil-rising-2015})
  \[\mathcal{X}:=\rproj \left( p^{*}\left( \bigoplus_{n\ge 0} \mathscr{I}^n \right) \right)\cong 
    \rproj \left( \bigoplus_{n\ge 0} \mathscr{I}^n \right)\times_X X'= \overline{X}\times_W W'\]

  The functor \(p^{*}\) commute with colimits, 
  so \(p^{*} \left( \bigoplus_{n\ge
      0} \mathscr{I}^n \right)=\bigoplus_{n\ge 0}
  p^{*}\left(\mathscr{I}^n\right)\). Since \(Y\) is flat over \(W\) the ideal sheaf
  \(\mathscr{I}'\) is \(p^{*} \left( \mathscr{I} \right)\), \citep
  {vakil-rising-2015}
  Then for each \(n\ge 0\) \(\left( \mathscr{I}' \right)^n=\left(p^{*}\left(
      \mathscr{I} \right) \right)^n\) but \(\left(p^{*}\left(
      \mathscr{I} \right) \right)^n = p^{*}\left(
    \mathscr{I}^n \right)\)
  \citep
  {vakil-rising-2015}.
  Hence \(\mathcal{X}=\overline{X'}\).\\


  Now if \(X\) is flat over \(W\) 
  the sheaf \(\mathscr{I}\)
  is also flat over \(W\).
  Since the sheaf \(\mathscr{I}/\mathscr{I}^2\cong \mathscr{I}\otimes_{\mathscr{O}_X} \mathscr{O}_X/\mathscr{I}\) is flat over \(W\),
  the sheaf \(\mathscr{I}^2\) is flat over \(W\) and by induction so is
  \(\mathscr{I}^n\) for all \(n\ge 1\).
\end{proof}

\begin{corollary}\label{coro:morphism-between-blow-ups}
  Consider the following Cartesian diagram.
  \[
    \begin{tikzcd}
      Y' \rar{b} \dar[swap]{f}  & X' \ar{r}{\beta} \dar{g} 
      & Y' \ar{d}{f} \\
      Y \rar[swap]{a}  &  X \rar[swap]{\alpha} & Y\\
    \end{tikzcd}
  \]
  
  
  Let \(\apl{\alpha\!:\!X}{Y}\) 
  be a separated morphism and \(\apl{a \!:\!Y }{X }\)
  a section of \(\alpha\).
  By Lemma~\ref{lem:imatge section is closed} the morphism \(a\)
  is a closed embedding. Assume that \(a\) corresponds to a finite type sheaf
  of ideals. Then \(\beta\) is separated and \(b\)
  is a closed embedding corresponding to a finite type sheaf of ideals.\\

  Consider the blow ups \(\apl{\bl_a\!:\! \Bl_a(X)}{X}\)
  and \(\apl{\bl_b\!:\!\Bl_{b}(X')}{X'}\).
  If the morphism \(\alpha\)
  is flat and finitely presented, then there is a unique morphism
  \(\apl{\theta(g)\!:\!\Bl_b(X')}{\Bl_a(X)}\)
  which makes the following diagram Cartesian.
  \[
    \begin{tikzcd}
      \Bl_b(X')    \dar[dashed,swap]{\theta(g)}   \rar{\bl_b} & X' \dar{g} \\
      \Bl_a(X) \rar[swap]{\bl_a} & X  \\
    \end{tikzcd}
  \]
\end{corollary}
\begin{proof}
  The scheme \(X'\) is the base change of \(X\) by \(\apl{Y'}{Y}\) and the identity
  morphism \(\alpha\circ a\) is flat.
\end{proof}

\begin{corollary}\label{coro:blowup-of-a-family-is-a-family}
  Given a section \(\apl{\sigma\!:\! B}{\mathcal{S}}\) of the family
  \(\apl{\pi\!:\!\mathcal{S}}{B}\), the morphism
  \(\apl{\pi\circ\bl_{\sigma}\!:\!\Bl_{\sigma}(\mathcal{S})}{B}\) is a family. 
\end{corollary}
\begin{proof}
  The morphism \(\pi\circ \bl_{\sigma}\) is separated and
  surjective. By Lemma~\ref{lem:flat-and-baseChange-blowup} it is flat and
  commutes with any base change \(\apl{B'}{B}\), so its
  generic fibre \(\eta\) is integral because \(\eta\) is a blow up of the generic
  fibre of \(\pi\). 
\end{proof}

\begin{definition}
  An \emph{ordered \(r\)-relative
    cluster of} \(\apl{\pi\!:\! \mathcal{S}}{B}\)
  is a finite sequence of morphisms \((\sigma_0,\dots,\sigma_{r})\)
  such that for all \(n =0,\dots,r\)
  each \(\sigma_n\)
  is a section of the family
  \(\apl{\pi_n\!:\!  \mathcal{S}_n}{B}\) defined recursively as follows.
  
  The family \(\pi_0\)
  is \(\pi\),
  so \(\mathcal{S}_0\)
  is \(\mathcal{S}\).
  For \(n= 0,\dots,r-1\)
  the scheme \(\mathcal{S}_{n+1}\)
  is \(\Bl_{\sigma_n}(\mathcal{S}_n)\)
  and the family \(\pi_{n+1}\)
  is the composition \(\pi_n\circ \bl_{\sigma_n}\) (\(\pi_n\) is a family by
  Corollary~\ref{coro:blowup-of-a-family-is-a-family}).
\end{definition}


\begin{definition}
  An \emph{ordered \(r\)-relative
    clusters family of} \(\apl{\pi\!:\! \mathcal{S}}{B}\)
  (or an \(r\)-rcf)
  is a couple \((C,\beta_{\bullet})\)
  consisting of a scheme \(C\)
  and a finite sequence \(\beta_{\bullet}\)
  of morphisms \(\beta_{0},\dots,\beta_{r}\)
  such that for all \(n=0,\dots,r\)
  each \(\beta_{n}\)
  is a section of the morphism
  \(\apl{\pi_{n}^{C}\!:\! \mathcal{S}_{n}^{C}}{C\times B}\)
  defined recursively as follows.

  The morphism \(\pi_0^C\)
  is \(\apl{\Id_C\times \pi\!:\! C\times\mathcal{S}}{C\times B}\).
  For \(n=0,\dots,r-1\)
  the scheme \(\mathcal{S}_{n+1}^{C}\)
  is \(\Bl_{\beta_n}(\mathcal{S}^C_n)\)
  and the morphism \(\pi^C_{n+1}\)
  is the composition \(\pi^C_n\circ \bl_{\beta_n}\),
  as shown in the following commutative diagram.
  \[
    \begin{tikzcd}
      \mathcal{S}^C_{n+1}:= \Bl_{\beta_n}\left( \mathcal{S}^C_n \right) \drar[swap]{\pi^C_{n+1}} \rar{\bl_{\beta_n}} &
      \mathcal{S}^C_{n} \dar[swap, near start]{\pi^C_n}
      \\ 
      & C\times B \uar[swap, bend right]{\beta_n} \\
    \end{tikzcd}
  \]
\end{definition}

Given an \(r\)-rcf
\((C,\beta_{\bullet})\)
of \(\pi\),
if the morphism \(\apl{\pi^C_0\!:\!  C\times \mathcal{S}}{C\times B}\)
is a family (e.g. \(C=\s(K)\))
then the sequence \((\beta_0,\dots,\beta_r)\)
is just an ordered \(r\)-relative cluster of \(\pi^C_0\).

\begin{remark}
  For an \(r\)-rcf
  \((C,\beta_{\bullet})\)
  of \(\pi\)
  the blow up \(\apl{\bl_{\beta_r}\!\!\!:\mathcal{S}^C_{r+1}}{\mathcal{S}^C_r}\)
  and the morphism \(\pi^C_{r+1}:=\pi^C_r\circ \bl_{\beta_r}^C\)
  exist. But the morphism \(\pi^C_{r+1}\) may admit no sections.
\end{remark}

\begin{notation}
  Given an \(r\)-rcf
  \((C,\beta_{\bullet})\)
  of \(\pi\)
  with \(k\le r\)
  and \(\beta_{\bullet}= (\beta_0,\dots,\beta_{r})\),
  we denote by \(\beta_{\bullet}|_k\)
  the subsequence \((\beta_0,\dots,\beta_{k})\),
  so \((C,\beta_{\bullet}|_k)\) is a \(k\)-rcf.
\end{notation}

\begin{remark}
  Given an \(r\)-rcf
  \((C,\beta_{\bullet})\)
  with \(k\le r\),
  the target schemes of \(\beta_i\) and \((\beta_{\bullet}|_k)_i\) are the same.
\end{remark}

\begin{definition}
  Let \((C,\beta_{\bullet})\)
  and \((C',\beta'_{\bullet})\)
  be \(r\)-rcf
  of \(\apl{\pi \!:\!  \mathcal{S}}{ B}\)
  and \(\apl{f \!:\! C }{C' }\)
  a morphism. Given \(k\le r\),
  we call a morphism \(\apl{g \!:\! \mathcal{S}^C_k }{\mathcal{S}^{C'}_k }\)
  a \(k\)-\emph{section lift} of \(f\) if the following diagram commutes.
  
  \begin{equation}
    \begin{tikzcd}
      C\times B %
      \rar{\beta_k} \dar[swap]{f\times \Id_B} & \mathcal{S}_{k}^C \dar{g} \\
      C'\times B \rar[swap]{\beta'_k} & \mathcal{S}_k^{C'} \\
    \end{tikzcd}\label{eq:1}
  \end{equation}
  Given \(k\le r+1\),
  we call the morphism \(g\)
  the \(k\)-\emph{projection lift} of \(f\) if the following diagram is
  Cartesian. 
  \begin{equation}
    \begin{tikzcd}
      \mathcal{S}_{k}^C \dar[swap]{g} \rar{\pi_{k}^{C}}  & C\times B
      \dar{f\times \Id_{B}}  \\
      \mathcal{S}_k^{C'}
      \rar[swap]{\pi_{k}^{C'}} & C'\times B  \\
    \end{tikzcd}\label{eq:2}
  \end{equation}
\end{definition}

\begin{remark}\label{rmk:initial cartesian diagram}
  If the morphism \(\apl{g \!:\! \mathcal{S}^C_k }{\mathcal{S}^{C'}_k }\)
  is the \(k\)-projection
  lift of \(f\)
  then it is a \(k\)-section
  lift of \(f\) if and only if the diagram~\eqref{eq:1} is Cartesian.
\end{remark}

The following Proposition is an immediate consequence of
Corollary~\ref{coro:morphism-between-blow-ups} and Remark~\ref{rmk:initial cartesian diagram}.

\begin{proposition}\label{prop:tetha of g}
  Let \((C,\beta_{\bullet})\)
  and \((C',\beta'_{\bullet})\)
  be \(r\)-rcf
  of \(\apl{\pi \!:\!  \mathcal{S}}{ B}\)
  with \(k\le r\)
  and \(\apl{f \!:\! C }{C' }\)
  a morphism.  If a morphism
  \(\apl{g \!:\! \mathcal{S}^C_k }{\mathcal{S}^{C'}_k }\)
  is both the \(k\)-projection
  and a \(k\)-section
  lift of \(f\)
  then there is a unique morphism
  \(\apl{\theta(g) \!:\! \mathcal{S}^C_{k+1} }{\mathcal{S}^{C'}_{k+1} }\)
  which is the \((k+1)\)-projection lift of \(f\).
\end{proposition}


\begin{definition}\label{def:r-rcf-morphism}
  A morphism between two \(r\)-rcf
  \((C,\beta_{\bullet})\)
  and \((C',\beta'_{\bullet})\)
  of \(\apl{\pi \!:\! \mathcal{S} }{B }\),
  or an \emph{\(r\)-rcf-morphism},
  is a morphism of schemes \(\apl{f\!:\! C}{C'}\)
  such that for all \(n=0,\dots,r\) the 
  morphism \(\apl{f_{n}\!:\! \mathcal{S}_{n}^{C}}{\mathcal{S}_{n}^{C}}\),
  defined recursively as follows, is an \(n\)-section lift of \(f\).

  The morphism \(f_{0}\)
  is \(f\times\Id_{\mathcal{S}}\),
  which is the 0-projection lift of \(f\).
  By hypothesis \(f_0\)
  is a 0-section lift of \(f\).
  By Proposition~\ref{prop:tetha of g} there is a unique morphism
  \(\apl{f_1:=\theta(f_0) \!:\! \mathcal{S}_1^C}{ \mathcal{S}_1^{C'}}\)
  which is the 1-projection lift of \(f\).
  Again by hypothesis \(f_1\)
  is a 1-section lift of \(f\) and we can iterate the process.
\end{definition}

Clearly by definition the \(r\)-rcf-morphisms are stable under composition and
the identity is an \(r\)-rcf-morphism from an \(r\)-rcf to itself. So the
set of all \(r\)-rcf of \(\pi\) is a category with the \(r\)-rcf-morphisms as
morphisms.

\begin{remark}\label{rmk:extra morphism, Cartesian square of base, extra
    C'-ssf}
  For each \(n=0,\dots,r\)
  the morphism \(\apl{f_{n}\!:\! \mathcal{S}_{n}^{C}}{\mathcal{S}_{n}^{C'}}\)
  is both the \(n\)-projection and an \(n\)-section lift of \(f\).
  There is a unique morphism
  \(\apl{f_{r+1}\!:\! \mathcal{S}^C_{r+1}}{\mathcal{S}^{C'}_{r+1}}\)
  which is the \((r+1)\)-projection lift of \(f\).
  Furthermore, if \((C,\beta_{\bullet})\)
  is an \((r+1)\)-rcf of \(\pi\), the following diagram commutes.
  \[
    \begin{tikzcd}
      C\times B \drar[swap]{\Id_{C\times B}} \rar{\beta_{r+1}} &
      \mathcal{S}^{C}_{r+1} \rar{f_{r+1}} \dar[swap, near start]{\pi^{C}_{r+1}}
       &
      \mathcal{S}^{C'}_{r+1} \dar{\pi^{C'}_{r+1}}    \\
      & C\times B \rar[swap]{f\times \Id_{B}} & C'\times B \\
    \end{tikzcd}
  \]
  i.e. \(\left(C,(f_{r+1}\circ \beta_{r+1}),f\right)\)
  is a \(C'\)-ssf of \(\pi^{C'}_{r+1}\).
\end{remark}

\begin{notation}
  We denote \(\apl{f\!:\! (C,\beta_{\bullet})}{(C',\beta'_{\bullet})}\)
  a morphism \(\apl{f\!:\! C}{C'}\)
  which is an \(r\)-rcf-morphism from \((C,\beta_{\bullet})\) to \((C',\beta'_{\bullet})\).
\end{notation}

\begin{definition}\label{def:Urcf-of-pi}
  Let \((\Cl_r,\alpha^r_{\bullet})\)
  be an \(r\)-rcf
  of \(\apl{\pi\!:\! \mathcal{S}}{B}\).
  We call \((\Cl_r,\alpha^r_{\bullet})\)
  an \emph{Universal \(r\)-relative
    clusters family} (or \(r\)-Urcf)
  of \(\pi\)
  if it satisfies the following universal property: for any \(r\)-rcf
  \((C,\beta_{\bullet})\)
  of \(\pi\)
  there is a unique morphism \(\apl{f\!:\! C}{\Cl_r}\)
  which is an \(r\)-rcf-morphism
  \(\apl{f\!:\! (C,\beta_{\bullet})}{(\Cl_r,\alpha^r_{\bullet})}\).
\end{definition}

If an \(r\)-Urcf
of \(\apl{\pi\!:\! \mathcal{S}}{B}\)
exists, by abstract nonsense, it is uniquely determined, up to unique
isomorphism. We denote it by \((\Cl_r,\alpha^r_{\bullet})\), for each \(n=0,\dots,r+1\)
we denote the scheme \(\mathcal{S}^{\Cl_r}_n\) by \(\mathcal{S}^r_n\) and the morphism
\(\pi^{\Cl_r}_n\) by \(\pi^r_n\). For each
\(n=0,\dots,r\) we denote the blow
up \(\bl_{\alpha^r_n}\) by \(\bl_{n+1}^r\).

\begin{remark}\label{equivalence sf - rcf}
  A couple \((Y,\rho)\)
  is an sf of \(\apl{\pi \!:\!  \mathcal{S}}{ B}\)
  if and only if the couple \(\left(Y, \left(\Id_Y\hspace{-2.5pt}\times_Y\:\rho\right)\right)\)
  is a \(0\)-rcf
  of \(\pi\).
  The morphism \(\apl{(\Id_Y\hspace{-2.5pt}\times_Y\rho)\!:\! Y\times B}{Y\times\mathcal{S}}\)
  is a section of \((\Id_Y\hspace{-2.5pt}\times\pi)\) which is separated and surjective. 
  So, by Lemma~\ref{lem:imatge section is closed} the image of
  \((\Id_Y\hspace{-2.5pt}\times_Y\rho)\) is closed and, if the Usf \((X,\psi)\)
  and the \(0\)-Urcf
  \((\Cl_0,\alpha^0_0)\)
  of \(\pi\) exist, \(X\cong \Cl_0\) and \(\alpha^0_0\cong \Id_X\hspace{-2.5pt}\times_X \psi\).
\end{remark}

\begin{notation}\label{not:pr-and-fr}
  If the \(r\)-Urcf
  \(\left(\Cl_{r},\alpha^{r}_{\bullet}\right)\)
  of \(\apl{\pi\!:\! \mathcal{S}}{B}\)
  exists then the couples \(\left(\Cl_{r},\alpha^{r}_{\bullet}|_{r-1}\right)\)
  and
  \(\bigl(\Cl_{r},\left(\alpha^{r}_0,\dots, \alpha^{r}_{r-2},\bl^{r}_{r}\circ
    \alpha^{r}_{r}\right)\bigr)\)
  are \((r-1)\)-rcf
  of \(\pi\). So, if the \((r-1)\)-Urcf
  \(\left(\Cl_{r-1},\alpha^{r-1}_{\bullet}\right)\)
  of \(\pi\) exists too, there are two \((r-1)\)-rcf-morphisms
  \[
    \aplllarga{p^{r-1}}{\left(\Cl_{r}, \alpha^{r}_{\bullet}|_{r-1}\right)}{
      \left(\Cl_{r-1},\alpha^{r-1}_{\bullet} \right)}
  \]
  \[
    \aplllarga{f^{r-1}}{\bigl(\Cl_{r}, \left(\alpha_0^{r},\dots,
        \alpha_{r-2}^{r},\bl^{r}_{r} \circ\alpha^{r}_{r}\right)\bigr)}{
      \left(\Cl_{r-1},\alpha^{r-1}_{\bullet}\right)}.
  \]
\end{notation}

\begin{remark}\label{rmk:projection-of-projection-is-equal-to-the-projection-of-fr}
  If for each \(n=(r-1),r,(r+1)\)
  the \(n\)-Urcf
  of \(\apl{\pi\!:\! \mathcal{S}}{B}\)
  exists, then \(p^{r-1}\circ p^{r} = p^{r-1}\circ f^{r}\)
  as \((r-1)\)-rcf-morphisms.
\end{remark}

In the forthcoming sections we always consider an \((r+1)\)-Urcf
\(\Cl_{r+1}\)
to be a \(\Cl_r\)-scheme
with structure morphism \(p^{r}\);
in particular when we consider some fibre product of some \((r+1)\)-Urcf
it will be understood with respect to the morphism \(p^{r}\).\\

Given the family \(\apl{\pi \!:\! \mathcal{S} }{B }\)
and an integer \(r\ge 0\)
there is the contravariant functor \(\apl{\cl^r\!:\! K \mbox{-}\Sch}{\Set}\)
corresponding to the parameter space problem of the ordered \(r\)-relative
clusters of \(\pi\)
defined as follows. It sends each scheme \(C\)
to the set of sequences of morphisms
\[
  \cl^r C:=\{\apl{\beta_{\bullet}\!:\! C\times B}{\mathcal{S}^C_{\bullet}} \mbox{ such that }
  (C,\beta_{\bullet}) \mbox{ is an \(r\)-rcf of }\pi\}.
\]
Given a morphism \(\apl{f\!:\! C}{C'}\)
and a sequence of morphisms \(\beta'_{\bullet}\in \cl^r C'\)
the image \(\beta_{\bullet} :=\cl^r f(\beta'_{\bullet})\)
is a sequence of morphisms such that \((C,\beta_{\bullet})\)
is an \(r\)-rcf
of \(\pi\)
and \(f\)
is an \(r\)-rcf-morphism
between \((C,\beta_{\bullet})\)
and \((C',\beta'_{\bullet})\).
For \(n=0,\dots,r\)
we construct recursively each section \(\beta_n\)
of the morphism \(\pi_{n}^{C}\)
and each morphism
\(\apl{f_n\!:\! \mathcal{S}_n^{C}}{\mathcal{S}_n^{C'}}\) at the same time as
follows.

The morphism \(f_0\)
is \(f\times \Id_{\mathcal{S}}\), which is the 0-projection lift of \(f\).
We define \(\apl{\beta_0\!:\! C\times B}{\mathcal{S}_0^{C}}\)
as the section given by Lemma~\ref{lem:first} applied to the diagram~\eqref{eq:2}
with \(g=f_0\)
and the section \(\beta'_0\).
Now \(\apl{f \!:\! (C,\beta_0)}{ (C',\beta'_0)}\)
is a 0-rcf-morphism and \(f_0\)
is a 0-section lift of \(f\).
By Proposition~\ref{prop:tetha of g} there is a unique morphism
\(\apl{f_1:=\theta(f_0) \!:\! \mathcal{S}^C_1}{\mathcal{S}^{C'}_1}\)
which is the 1-projection lift of \(f\).
We can iterate the process to obtain \(\beta_n\) and
\(f_{n+1}\) for \(n=1,\dots,r\).\\

Given two \(r\)-rcf
\((C,\beta_{\bullet})\)
and \((C',\beta'_{\bullet})\)
of \(\apl{\pi\!:\! \mathcal{S}}{B}\),
in the definition of the functor \(\cl^r\) it was shown that a morphism
\(\apl{f\!:\! C}{C'}\)
is an \(r\)-rcf-morphism
\(\apl{f\!:\! (C,\beta_{\bullet})}{(C',\beta'_{\bullet})}\)
if and only if \(\beta_{\bullet}=\cl^r f(\beta'_{\bullet})\), which gives an explanation
for the (maybe surprising at first sight) definition of \(r\)-rcf-morphism.\\

\begin{definition}
  Given an \(r\)-rcf
  \((C,\beta_{\bullet})\)
  of \(\apl{\pi\!:\! \mathcal{S}}{B}\)
  and \(c\in C\),
  slightly abusing language we call the sequence of morphisms \((\underline{c}^0,\dots,\underline{c}^r):=\cl^r\im_{c}(\beta_{\bullet})\)
  the \emph{ordered \(r\)-relative
    cluster associated to \(c\)}.
\end{definition}

\begin{proposition}
  Given an \(r\)-rcf
  \((C,\beta_{\bullet})\)
  of \(\apl{\pi\!:\! \mathcal{S}}{B}\)
  and a closed \(c\in C\),
  the ordered \(r\)-relative
  cluster associated to \(c\) is an ordered \(r\)-relative cluster of \(\pi\).
\end{proposition}
\begin{proof}
  Clearly from the definition \(\underline{c}^n\)
  is a section of \(\pi^{\s(k(c))}_n\)
  and \(\mathcal{S}^{\s(k(c))}_{n+1}\)
  is the blow up of \(\mathcal{S}^{\s(k(c))}_n\)
  at \(\underline{c}^n(B^{c})\).
  Note that the schemes \(\mathcal{S}^{\s(k(c))}_0\)
  and \(B^{c}\)
  are respectively isomorphic to \(\mathcal{S}\)
  and \(B\), so \(\underline{c}^0\) can be identified with a section of \(\pi\).
\end{proof}

\begin{remark}\label{rmk:assosiation-rcf-and-s}
  For each \(n=0,\dots,r+1\),
  the scheme \(\mathcal{S}^{\s(k(c))}_n\)
  is isomorphic to the fibre over \(c\in C\) of \(\apl{\mathcal{S}_{n}^C}{C}\).
\end{remark}

Theorem~\ref{definicio2} is a natural generalisation of
Proposition~\ref{def:functor S._} and is proved essentially the same way.

\begin{theorem}\label{definicio2}
  A scheme \(C\)
  represents the functor \(\cl^r\)
  associated to \(\apl{\pi \!:\! \mathcal{S}}{B}\)
  if and only if \(C\) is the \(r\)-Urcf of \(\pi\).
\end{theorem}

\begin{corollary}
  If the \(r\)-Urcf
  \((\Cl_{r},\alpha^{r}_{\bullet})\)
  of the morphism \(\apl{\pi\!:\! \mathcal{S}}{B}\)
  exists, the following map is bijective.
  \[
    \begin{tikzpicture}[node
      distance=5cm]
      \node (1) {\( \Psi\hspace{.1cm}:\hspace{.1cm} \Cl_{r}(K)\)};
      \node [right of=1] (2) {\(\{\mbox{ordered \(r\)-relative
          clusters of }\pi\}\)};
      \node at (.5,-.7) {\( c\)};
      \node at (3.35,-.7) {\( (\underline{c}^0,\dots,\underline{c}^{r})\)};
      \node at (.95,-.7) (3) {}; \node at (2.24,-.7) (4) {};

      \path[->,font=\scriptsize,-stealth] (1) edge (2);
      \path[font=\scriptsize,|-stealth] (3) edge (4);
    \end{tikzpicture}
  \]
\end{corollary}

\begin{theorem}
  \label{thr:existence of Cl_r}
  Given the family \(\apl{\pi\!:\! \mathcal{S}}{B}\),
  if \(B\)
  is projective and \(\mathcal{S}\)
  is piecewise quasiprojective then its \(r\)-Urcf
  \((\Cl_{r},\alpha^r_{\bullet})\)
  exists and the scheme \(\Cl_r\)
  is piecewise quasiprojective.
\end{theorem}

\begin{proof}
  We work by induction over \(r\). For \(r=0\), Remark~\ref{equivalence sf - rcf}
  and Theorem~\ref{thm:exist-Usf} says that the
  0-Urcf of \(\pi\)
  exists, it is its Usf which is piecewise quasiprojective.\\

  Assume that the \((r-1)\)-Urcf
  \((\Cl_{r-1},\alpha^{r-1}_{\bullet})\)
  of \(\pi\)
  exists with \(\Cl_{r-1}\)
  piecewise quasiprojective.

  The scheme \(\mathcal{S}^{r-1}_{r}\) 
  is piecewise quasiprojective. By Corollary~\ref{thm:exist-W-Ussf} the
  \(\left( \Cl_{r-1} \right)\)-Ussf
  of \(\apl{\pi^{r-1}_{r}\!:\! \mathcal{S}^{r-1}_r}{\Cl_{r-1}\times B}\)
  exists. We denote it by \((\Cl_r,\gamma_r,p^{r-1})\),
  note that \(\Cl_r\) is piecewise quasiprojective.

  We will see that the scheme \(\Cl_r\) represents the functor \(\cl^r\).
  First we will construct a morphism of functors \(\nu\) from \(\cl^r\) to
  \(\Hom_{Sch}(\_,\Cl_r)\) and then its inverse \(\mu\).\\

  Given a scheme \(Z\) and a \(\beta_{\bullet}\in\cl^r(Z)\), the couple
  \((Z,\beta_{\bullet}|_{r-1})\) is an \((r-1)\)-rcf of \(\pi\). By the universal property
  of the \((r-1)\)-Urcf \((\Cl_{r-1},\alpha^{r-1}_{\bullet})\) of \(\pi\)
  there is a unique morphism \(\apl{f\!:\!Z}{\Cl_{r-1}}\) such that
  \(\beta_{\bullet}|_{r-1}=\cl^{r-1}f(\alpha^{r-1}_{\bullet})\). Since \((Z,\beta)\) is
  an \(r\)-rcf of \(\pi\), by Remark~\ref{rmk:extra morphism, Cartesian square of base, extra
    C'-ssf} the triplet \((Z,(f_r\circ\beta_r),f)\) is a \((\Cl_{r-1})\)-ssf of
  \(\pi^{r-1}_r\). By the universal property of the \(\left( \Cl_{r-1} \right)\)-Ussf \((\Cl_r,\gamma_r,p^{r-1})\)
  of \(\pi^{r-1}_{r}\) there is a unique morphism \(\apl{\nu_Z(\beta_{\bullet}) \!:\!Z}{\Cl_r}\) such
  that \(f_r\circ\beta_r=\gamma_r\circ(\nu_Z(\beta_{\bullet})\times \Id_B)\). Then \(\nu\) is clearly a morphism of functors.\\

  Given a scheme \(Z\) and a morphism \(\apl{h\!:\!Z}{\Cl_r}\), we define
  \((\beta_0,\dots,\beta_{r-1})=\cl^{r-1}(p^{r-1}\circ
  h)(\alpha^{r-1}_{\bullet})\). Observe that \(\mathcal{S}^Z_r\) is
  \(\Bl_{\beta_{r-1}}(\mathcal{S}^Z_{r-1})\) which endowed with the
  morphism \(\apl{(p^{r-1}\circ h)_r\!:\!\mathcal{S}^Z_r}{\mathcal{S}^{r-1}_r}\).
  By Remark~\ref{rmk:extra morphism, Cartesian square of base,
    extra C'-ssf} the scheme \(\mathcal{S}^Z_r\)
  is the fibre product
  \((Z\times B)\times_{(\Cl_{r-1}\times B)} \mathcal{S}^{r-1}_r\).
  We use this fact to define the morphism
  \(\apl{\beta_r\!:\!Z\times B}{\mathcal{S}^Z_r}\) as
  \[
    \beta_{r}:=\Id_{(Z\times B)}\times_{(\Cl_{r-1}\times B)} (\gamma_r\circ
    (h\times \Id_B))
  \] which is clearly well defined. Then \(\mu_z(h):=(\beta_0,\dots,\beta_r)\)
  and \(\mu\) is a morphism of functors.\\

  We leave to the reader to check that \(\nu\) and \(\mu\) are inverse one to each other.
\end{proof}

\begin{remark}\label{rmk:definition-alpha-r-r}
Given the \((r-1)\)-Ussf \((\Cl_{r-1},\alpha^{r-1}_{\bullet})\) of \(\pi\) and the scheme \(\Cl_r\), the sequence of
morphisms \(\alpha_{\bullet}^r\) is \(\mu_{\Cl_r}(\Id_{\Cl_r})\). In particular,
\[
     \alpha^r_{r}=\Id_{(\Cl_r\times B)}\times_{(\Cl_{r-1}\times B)} \gamma_r
\] and \(\gamma_r=p^{r-1}_r\circ\alpha^r_r\).
\end{remark}
\section{\(\Cl_{r+1}\) as a blowup scheme}\label{family blowups}
In this section we consider a  family \(\apl{\pi \!:\!  \mathcal{S}}{ B}\)
with \(\mathcal{S}\) quasiprojective and \(B\) projective and regular. By
Theorem~\ref{thr:existence of Cl_r} the \(r\)-Urcf \((\Cl_r,\alpha^r_{\bullet})\)
of \(\pi\) exists for all integer \(r\ge 0\).

Note from the proof of Theorem~\ref{thr:existence of Cl_r} that for all integer \(r\ge
0\) the triplet 
\((\Cl_r,\gamma_r,p^{r-1})\) is the \(\left(\Cl_{r-1}\right)\)-Ussf of
\(\apl{\pi^{r-1}_r\!:\!\mathcal{S}^{r-1}_r}{\Cl_r\times B}\), where
\(\apl{\gamma_r\!:\!\Cl_r\times B}{\mathcal{S}^{r-1}_r}\) is \(\gamma_r=p^{r-1}_r\circ\alpha^r_r\). 

\begin{notation}
  Given an \(r\ge 0\) and a \(e\in\Cl_r\) we fix the following notation.
  \begin{itemize}
  \item From the \(\left(\Cl_{r-1}\right)\)-Ussf \((\Cl_r,\gamma_r,p^{r-1})\) of
    \(\pi_r^{r-1}\), the \(\left(\Cl_{r-1}\right)\)-split section
    \(\apl{\tilde{e}\!:\!B^e}{\left(\mathcal{S}_r^{r-1}\right)_{p^{r-1}(e)} }\)
    associated to \(e\).
  \item From the \(r\)-Urcf \((\Cl_{r},\alpha^{r}_{\bullet})\) of \(\pi\), the
    ordered \(r\)-relative cluster \((\underline{e}^0,\dots,\underline{e}^{r})\)
    associated to \(e\), where
    \(\apl{\underline{e}^{r} \!:\! B^{e} }{\left(\mathcal{S}_{r}^{r}\right)_{e}}\).
  \item From the sf \((\Cl_{r},\gamma_{r})\) of
    \(\overapl{\mathcal{S}^{r-1}_{r} }{\Cl_{r-1}\times
      B}{\pi^{r-1}_{r}}\rightarrow B\) there is the section
    \(\apl{\underline{e} \!:\! B^{e}}{\mathcal{S}^{r-1}_{r}}\) associated to
    \(e\).
  \end{itemize}
\end{notation}


\begin{remark}\label{rmk:relative-section-equal-to-last-cluster}
  For integers \(n,m\ge 0\) and \(d\in \Cl_n\) we denote 
  \(\apl{(i_m^n)_{d} \!:\!(\mathcal{S}^n_m)_{d} }{\mathcal{S}^n_m}\) the
  inclusion. The relation between \(\underline{e}\), \(\tilde{e}\) and
  \(\underline{e}^{r}\)
  are
  \[
    \underline{e}=\left(i_{r}^{r-1}\right)_{p^{r-1}(e)}\circ\tilde{e}
    =p^{r-1}_{r}\circ \left(i_{r}^{r}\right)_{e}\circ\underline{e}^{r}.
  \]
\end{remark}

Now fix an integer \(r\ge 0\) for the rest of the section.
Consider \(\Cl_r^2:=\Cl_r\times_{\Cl_{r-1}} \Cl_r\) with the projections
\(\apl{q1,q2\!:\! \Cl_r^2}{\Cl_r}\)
over the first and the second factor respectively and \(\Delta_r\) its diagonal
morphism. Observe that
\(p^{r-1}\circ q1= p^{r-1}\circ q2\).

We call \(E\)
the exceptional divisor of the blow up
\(\apl{\bl_{r+1}^r \!:\! \mathcal{S}^r_{r+1} }{\mathcal{S}^r_{r} }\) (see Definition~\ref{def:Urcf-of-pi}),
whose centre is the image of \(\alpha^r_{r}\),
a closed subscheme (note that \(\alpha_r^r\)
is a section so its image is closed).\\

Our goal is to describe part of the components of \(\Cl_{r+1}\) as a blow up of
\(\Cl_r\times_{\Cl_{r-1}}\Cl_r\) in analogy with
Kleiman~\citep{kleiman-multiple-point-1981}.
Observe that the component of \(\Cl_{r+1}\) such that for all its points \(b\) the image of
\(\underline{b}\) is contained in \(E\) can not be emerge from this blow up.\\

For the following definition note from the proof of Theorem~\ref{thr:existence of
  Cl_r} that the scheme \(\mathcal{S}^r_r\)
is the fibre product
\(\left( \Cl_r\times B \right)\times_{(\Cl_{r-1}\times B)}\mathcal{S}^{r-1}_r\).

\begin{definition}\label{def:morphism-rho-and-F}
  We define \(\apl{\rho\!:\! \Cl_r^2\times B}{\mathcal{S}^r_r}\)
  as the product morphism,
  \[
    (q1\times \Id_B)\times_{(\Cl_{r-1}\times B)} \left( \gamma_{r}\circ (q2\times
      \Id_B)\right)
  \]
  which is clearly well defined and makes the following diagram commute.
  \[
    \begin{tikzcd}
      & \Cl_r^2\times B \dar[dashed]{\rho} \dlar[swap]{q1\times \Id_B}
      \drar{\gamma_r\circ (q2\times \Id_B)}
      \\
      \Cl_r\times B & \mathcal{S}^{r}_{r} \rar[swap]{p^{r-1}_{r}} \lar{\pi_{r}^r}
      &  \mathcal{S}^{r-1}_{r}   \\
    \end{tikzcd}
  \]
  We also define the morphism \(\apl{F\!:\!\Cl_{r+1}}{\Cl_r^2}\) as the product \(p^{r}\times_{\Cl_{r-1}}f^{r}\)
  which is well defined by Remark~\ref{rmk:projection-of-projection-is-equal-to-the-projection-of-fr}.
\end{definition}


Explicitly, over the closed points, \(\rho\)
sends each \((c,d,t)\)
to \((c,\underline{d}^{r}(t))\), which is possible because
\(p^{r-1}(c)=p^{r-1}(d)\).

Given a closed \(b\in\Cl_{r+1}\)  the pair of ordered \(r\)-relative
clusters associated to the image \(F(b)=(c,d)\in\Cl_{r}^2\) is \(\left(
  (\underline{b}^0,\dots,\underline{b}^{r-1},\underline{b}^r),(\underline{b}^0,\dots,\underline{b}^{r-1},\bl^r_{r+1}\circ
  \underline{b}^{r+1}) \right)\).

\begin{lemma}\label{lem:diagrama commutatiu gamma}
  The following diagram commutes.
  \[
    \begin{tikzcd}
      \Cl_{r+1}\times B \rar{\gamma_{r+1}} \dar[swap]{F\times \Id_B} & \mathcal{S}^{r}_{r+1} \dar{\bl_{r+1}^r} \\
      \Cl_r^2\times B \rar{\rho}  &  \mathcal{S}^r_r  \\
    \end{tikzcd}
  \]
\end{lemma}
\begin{proof}
  It is an immediate consequence of Remark~\ref{rmk:definition-alpha-r-r} and the definitions of \(F\) and \(\rho\) .
\end{proof}
\begin{remark}\label{rmk:rho-deltaR=alpha}
  Observe that \(\rho\circ (\Delta_r\times \Id_{B})=\alpha^r_{r}\), so the centre
  of the blow up \(\apl{\bl^r_{r+1}\!:\!\mathcal{S}^r_{r+1}}{\mathcal{S}^r_r}\) is \(\rho(\Delta_r(\Cl_r)\times B)\). 
\end{remark}

\begin{lemma}\label{lem:clr2-as-universal-split}
  The triplet \((\Cl_r^2,\rho, q1)\)
  is the \(\Cl_r\)-Ussf
  of \(\apl{\pi^r_r \!:\! \mathcal{S}^r_{r} }{\Cl_r\times B }\).
\end{lemma}
\begin{proof}
  The schemes \(\Cl_r\) and \(\mathcal{S}^r_r\) are piecewise quasiprojective and
  \(B\) is projective. By Corollary~\ref{thm:exist-W-Ussf} the \(\Cl_r\)-Ussf
  \(\left( Y,\beta,f \right)\)
  of \(\pi^{r}_{r}\)
  exists.  The triplet \((\Cl_r^2,\rho, q1)\)
  is a \(\Cl_r\)-ssf
  of \(\pi^{r}_{r}\).
  So there is a unique morphism \(\apl{L \!:\! \Cl_r^2 }{Y }\)
  such that \(\rho=\beta\circ \left( L\times \Id_B \right)\)
  and \(q1=f\circ L\).
  
  In fact, the morphism \(L\)
  is an isomorphism, we are going to find its inverse
  \(\apl{H \!:\! Y }{\Cl_r^2 }\).
  By the universal property of the Cartesian square of the following diagram, to
  construct \(H\)
  such that \(H\circ L=\Id_{\Cl_r^2}\)
  it is enough to find a morphism \(\apl{g \!:\! Y }{\Cl_r }\)
  which makes the following diagram commute.
  \[
    \begin{tikzcd}[row sep=1.1em, column sep=1.35em]
      Y \ar[bend left=20]{drrr}{g} \ar[bend right=20, swap]{dddr}{f}  & & & \\
      & \Cl_r^2 \ar{rr}{q2} \ar{dd}{q1} \ar{ul}{L}  & & \Cl_r
      \ar{dd}{p^{r-1}} \\
      \\
      & \Cl_r \ar{rr}{p^{r-1}} & & \Cl_{r-1} \\
    \end{tikzcd}
  \]
  Furthermore, by the universal properties of the \(\Cl_r\)-Ussf
  \((Y,\beta,f)\)
  and the fibre product
  \(\mathcal{S}^r_r=\left( \Cl_r\times B \right)\times_{(\Cl_{r-1}\times
    B)}\mathcal{S}^{r-1}_r\)
  the equality \(L\circ H = Id_Y\)
  holds if \(p^{r-1}_{r}\circ\beta=\gamma_r\circ(g\times \Id_B)\).\\

  The triplet \(\left( \Cl_r,\gamma_r,p^{r-1} \right)\)
  is the \(\left( \Cl_{r-1} \right)\)-Ussf
  of the morphism
  \(\apl{\pi^{r-1}_{r} \!:\! \mathcal{S}^{r-1}_{r} }{\Cl_{r-1}\times B }\).
  The triplet \(\left( Y,(p^{r-1}_{r}\circ\beta),(p^{r-1}\circ f) \right)\)
  is a \(\left( \Cl_{r-1} \right)\)-ssf
  of \(\pi^{r-1}_{r}\).
  So there is a unique morphism \(\apl{g \!:\! Y }{\Cl_r }\)
  such that \(p^{r-1}_{r}\circ\beta=\gamma_r\circ(g\times \Id_B)\)
  and \(p^{r-1}\circ f=p^{r-1}\circ g\).
  Finally the triplet
  \(\left( \Cl_r^2,(p^{r-1}_{r}\circ\rho),(p^{r-1}\circ q2) \right)\)
  is a \(\left( \Cl_{r-1} \right)\)-ssf
  of \(\pi^{r-1}_{r}\)
  too, but by definition of \(\rho\),
  \(p^{r-1}_{r}\circ\rho = \gamma_r\circ(q2\times \Id_B)\),
  so by the universal property of the \(\left( \Cl_{r-1} \right)\)-Ussf
  \(\left( \Cl_r,\gamma_r,p^{r-1} \right)\)
  the morphism \(q2\)
  is the unique morphism satisfying this equality and \(q2=g\circ L\).
\end{proof}

\begin{notation}
  Given \(e\in\Cl^2_r\) we fix the following notation.
  \begin{itemize}
  \item From the sf \((\Cl_r^2,\rho)\) of \(\overapl{\mathcal{S}^r_r
    }{\Cl_r\times B }{\pi^r_r}\rightarrow B\) there is the section
    \(\apl{\underline{e}_2 \!:\! B^{e} }{\mathcal{S}^r_r }\) associated to \(e\)
    (with a subindex intended to highlight its difference to the \(\apl{\underline{c}\!:\!B^c}{\mathcal{S}^r_r}\) for a given \(c\in\Cl_r\)).
  \item From the \(\Cl_r\)-Ussf \((\Cl_r^2,\rho,q1)\)
    of \(\apl{\pi^r_r\!:\!\mathcal{S}^r_r}{\Cl_r\times B}\)
    the \(\Cl_r\)-split
    section \(\apl{\tilde{e}\!:\!B^e}{(\mathcal{S}^r_r)_{q1(e)}}\)
    associated to \(e\).
  \end{itemize}
\end{notation}
\begin{remark}\label{rmk:rel-clrs-sections}
  The relation between \(\underline{e}_2\) and \(\tilde{e}\) is \(\underline{e}_2=(i^r_r)_{q1(e)}\circ\tilde{e}\).
\end{remark}

Given a closed \((c,d)\in\Cl_r^2\) and \(t\in B^{(c,d)}\), by definition of \(\rho\), \(\underline{(c,d)}_2(t)=(c,\underline{d}^r(t))\).
Furthermore, given \(c\in \Cl_r\), \(\underline{c}=\underline{\Delta_r(c)}_{2}\).

\begin{proposition}\label{prop:relation-sections-cl(r+1)-clr2}
  Given a point \(b\in \Cl_{r+1}\),
  its image \(F(b)=e\in \Cl_r^2\)
  belongs to \(\Delta_r(\Cl_r)\)
  if and only if \(\underline{b}(B^{b})\subseteq E\),
  and given \(t\in B^b\), \(\underline{b}(t)\) belongs to \(E\) if and only if
  \[\underline{e}_2(t)\in\rho \left( \Delta_r(\Cl_r)\times B \right).\]
\end{proposition}

\begin{proof}
  By Lemma~\ref{lem:diagrama commutatiu gamma},
  \(\underline{e}_2=\bl_{r+1}^r\circ \underline{b}\).
  Then the proposition follows from Remark~\ref{rmk:rho-deltaR=alpha}.
\end{proof}



\begin{remark}\label{coro:observation-cl(r+1)-and-delta-r}
  Given \(b\in \Cl_{r+1}\),
  the section \(\underline{b}^{r+1}\)
  of the ordered \((r+1)\)-relative
  cluster \((\underline{b}^0,\dots,\underline{b}^{r+1})\)
  associated to \(b\)
  belongs to the first infinitesimal neighbourhood of section \(\underline{b}^r\),
  or \(\underline{b}^{r+1}\) is infinitely near to \(\underline{b}^r\), if and
  only if \(F(b)\in\Delta_r(\Cl_r)\).
\end{remark}

We denote \(V\) the open set \(\Cl_{r+1}\setminus F^{-1}(\Delta_r(\Cl_r))\) which
parameterize the ordered \((r+1)\)-relative clusters
\((\underline{c}^0,\dots,\underline{c}^{r+1})\) such that the section
\(\underline{c}^{r+1}\) is not contained in \(E\) (i.e. \(\underline{c}^{r+1}\) is
\emph{not} infinitely near to \(\underline{c}^r\)).

We denote \(I\) the closed set \(\rho^{-1}(\alpha^r_r(\Cl_r\times
B))=\rho^{-1}(\rho(\Delta_r(\Cl_r)\times B))\subseteq \Cl_r^2\times B\) whose
closed points are the closed \((c,d,t)\in \Cl^2_r\times B\) such that
\(\underline{c}^r(t)=\underline{d}^r(t)\).

We define the scheme \(J\) by the following Cartesian square,

\begin{equation}
  \begin{tikzcd}
    J \ar{d} \ar{r}  & I  \ar{d}  \\
    V \ar{r}{F|_V}  & \Cl_r^2             \\
  \end{tikzcd}\label{eq:diagram-J-I}
\end{equation}
where \(\apl{I }{\Cl_r^2 }\) is the restriction of the projection
\(\apl{\Cl_r^2\times B }{\Cl_r^2 }\) to \(I\).

\begin{remark}\label{remark: r+1-clusters -> adm pari of r-clusters}
  Let \(b\in V\) be a point with image \(F(b)=e\in \Cl_r^2\).
  Since \(\bl_{r+1}^r\circ \underline{b}=\underline{e}_2\) we have (via the blow up \(\bl_{r+1}^r\))
  \[
    \underline{b}(B^b)\setminus E \cong
    \underline{e}_2(B^e)\setminus \alpha_r^r \left(\Cl_r\times B \right).
  \] Hence the  strict transform of \(\underline{e}_2(B)\)
  by the blow up \(\bl^r_{r+1}\) is
  \[
    \overline{(\bl_{r+1}^r)^{-1} \left(\underline{e}_2(B^e)\setminus \alpha^r_r \left( \Cl_r\times B \right) \right)} = \overline{b(B^b)\setminus E}=b(B^b)
  \] where the last equality is because \(\underline{b}(B^b)\not\subseteq E\) since \(b\in V\).
\end{remark}

\begin{remark}  \label{rmk:iso-between-imatge-section-and-B}
  Given \(e\in \Cl_r^2\), the scheme
  \[
    \underline{e}_2(B^e)\cap \alpha^r_r(\Cl_r\times B)
  \]
  is isomorphic to the fibre \(I_e\)
  of \(\apl{I }{\Cl_r^2 }\).
\end{remark}

\begin{definition}
  Given \(e\in \Cl_r^2\),
  we call \(e\) 
  a pair of admissible \(r\)-relative
  clusters if
  \[
    \underline{e}_2(B^e)\cap \alpha^r_r(B)
  \]

  is an effective Cartier divisor of \(\underline{e}_2(B)\).
  If \(e\) is closed equal to \((c,d)\) we also call \(d\) admissible with
  respect to \(c\).
  
  We denote by \(\Cl^{\adm}_r\subseteq \Cl_r^2\)
  (resp. by \(\Cl^{\adm}_c\subseteq \Cl_r\))
  the set of all pairs of admissible \(r\)-relative
  clusters (resp. the set of all admissible \(r\)-relative
  clusters with respect to \(c\)).
\end{definition}

\begin{corollary}
  The set \(\Cl^{\adm}_r\) and for all closed \(c\in \Cl_r\),
  the set \(\Cl^{\adm}_c\) are constructible.
\end{corollary}

\begin{proof}
  By Remark~\ref{remark: r+1-clusters -> adm pari of r-clusters},
  \(\Cl^{\adm}_r=F(\Cl_{r+1})\setminus\Delta_r(\Cl_r)\)
  and \(\Cl^{\adm}_c=f^r(\Cl_{r+1})\setminus \{c\}\) (the definition
  of \(f^r\) is in Notation~\ref{not:pr-and-fr}). The claim follows from
  Chevalley's theorem \citep[ex. 3.19]{hartshorne-algebraic-1977}. 
\end{proof}

\begin{proposition}\label{flat proj over V}
  \(J\subseteq V\times B\)
  is a relative effective Cartier divisor on \((V\times B)/V\)
\end{proposition}

\begin{proof}
  The following diagram, where the vertical arrows are the projections, is a
  Cartesian square.
  \[
    \begin{tikzcd}
      V\times B \rar{F|_{V}\times \Id_B} \dar[swap]{} & \Cl^2_r\times B \dar{} \\
      V \rar[swap]{F|_V} & \Cl^2_r  \\
    \end{tikzcd}
  \]
  Hence the scheme \(J=V\times_{\Cl^2_r} I\)
  is a closed subscheme of \(V\times B\),
  the morphism \(J\rightarrow V\)
  is the (flat) projection \(\apl{V\times B }{V }\)
  restricted to \(J\) and the following diagram is Cartesian.
  \[
    \begin{tikzcd}
      J \rar{} \dar[hook]{} & I \rar{} \dar[hook]{} & \alpha^r_r(\Cl_r\times B) \dar[hook]{}  \\
      V\times B \rar[swap]{F|_V\times \Id_B} & Cl^2_r\times B \rar[swap]{\rho} & \mathcal{S}^r_r  \\
    \end{tikzcd}
  \] So \(J = (\rho\circ (F|_V\times \Id_B))^{-1}(\alpha^r_r(\Cl_r\times B))\)
  and by Lemma~\ref{lem:diagrama commutatiu gamma} the scheme
  \(J\) is just \(\left(\gamma_{r+1}|_{V\times B}\right)^{-1}(E)\).\\
  
  For every \(b\in V\),
  the fibre \(J_{b}\)
  of \(\apl{J }{V }\)
  is isomorphic (via \(\gamma_{r+1}|_{B^{b}}=\underline{b}\))
  to \(\underline{b}(B^{b})\cap E\) which is
  an effective Cartier divisor of \(\underline{b}(B^{b})\) (note that
  \(\underline{b}(B^b)\cap E \not= \underline{b}(B^b)\) by
  Proposition~\ref{prop:relation-sections-cl(r+1)-clr2} and \(b\in V\)).

  Then \(J_{b}\)
  is an effective Cartier divisor in the fibre \((V\times B)_{b}\)
  of \(\apl{V\times B}{V}\)
  (note that \((V\times B)_{b}\cong B^{b}\)) and the proposition follows from
  lemma  \citep
  {stacks-project}.
\end{proof}


\begin{theorem}\label{theorem of birrationaliti}
Given the family \(\apl{\pi \!:\!  \mathcal{S}}{ B}\) with \(\mathcal{S}\)
quasiprojective and \(B\) projective and regular, there is a stratification,
\(\apl{\bigsqcup_p \Cl^p}{\Cl_r^2}\) by connected locally closed subschemes, such that 
  \begin{itemize}
  \item[(1)] There are two kinds of stratum.  Type I with all 
    \(e\in \Cl^p\)
    an admissible pair of \(r\)-relative
    clusters and type II with all \(e\in \Cl^p\)
    not admissible pair of \(r\)-relative clusters.
  \item[(2)] \(\Delta_r(\Cl_r)\) is a stratum itself of type II.
  \item[(3)] \(V\)
    is isomorphic to the all disjoint union of type I strata, which we call
    \(\bigsqcup_{\adm}\Cl^p\).
  \end{itemize}
\end{theorem}

\begin{proof}
  The stratification \(\apl{\bigsqcup_p \Cl^p}{\Cl_r^2}\) of \(\Cl_r^2\) is the
  stratification by a locally finite set of connected locally closed subschemes
  given by the flattening stratification of the morphism \(\apl{I}{\Cl_r^2}\)
  (see \citep{nitsure-construction-2005} for the existence and the universal
  property of such flattening stratification) with the strata decomposed into
  connected components.\\

  (1) Given \(\Cl^p\)
  a stratum, by \citep[VI, Theorème 2.1 (i)]{grothendieck-fondements-1962} and
  \citep[3.4]{Kleiman-Picard-2005}, the set of \(e\in \Cl^p\) such that
  \(I_e\subseteq B^e\) is an effective Cartier divisor is both open and
  closed (here we use the hypothesis \(B\) regular). So, if there is \(e\in\Cl^p\)
  such that \(I_{e}\subseteq B^{e}\)
  is an effective Cartier divisor then so is \(I_{e'}\subseteq B^{e'}\)
  for all \(e'\in \Cl^p\) (note that \(\Cl^p\) is connected)
  and then Remark~\ref{rmk:iso-between-imatge-section-and-B} gives
  the statement.\\

  (2) For all \(e\in \Cl_r^2\),
  the dimension of the fibre \(I_{e}\)
  of the morphism \(\apl{I }{\Cl_r^2 }\)
  is equal to \(\dim(B)\)
  if and only if \(e\in \Delta_r(\Cl_r)\) and then \(I_e\cong B^e\).
  Therefore by the equidimensionality of the fibres of a flat
  morphism the closed set \(\Delta_r(\Cl_r)\) is a stratum (note from
  the definition of a family that \(B\) is irreducible).\\

  (3) By Proposition~\ref{flat proj over V} and the universal property of the
  flattening stratification there is a unique morphism
  \(\apl{F'\!:\! V}{\bigsqcup_p \Cl^p}\)
  which makes the following diagram commute.
  \[
    \begin{tikzcd}
      V \ar{rr}{F|_V} \ar{dr}[swap]{F'} &   & \Cl_{2}^r \\
      & \bigsqcup_p \Cl^p \ar[hook]{ur} & \\
    \end{tikzcd}
  \]
  Furthermore, by Remark~\ref{remark: r+1-clusters -> adm pari of r-clusters} and
  the universal property of \((\Cl_{r+1},\gamma_{r+1},p^r)\),
  \(F'(V) = \bigsqcup_{\adm}\Cl^p\).
  We call \(\apl{F_0\!:\! V}{\bigsqcup_{\adm}\Cl^p}\)
  the corestriction of \(F'\) which is surjective.
  
  We define \(\sqcup_{\adm} I^p\) by the following Cartesian square.
  \[
    \begin{tikzcd}
      \sqcup_{\adm}I^p \ar[hook]{r}  \ar{d} & I \ar{d} \\
      \sqcup_{\adm}\Cl^p \ar[hook]{r}  & \Cl_r^2 \\
    \end{tikzcd}
  \]
  The scheme \(\sqcup_{\adm}I^p\)
  is a subscheme of \(\sqcup_{\adm}\Cl^p\times B\).
  The morphism\(\apl{\sqcup_{\adm}I^p }{\sqcup_{\adm}\Cl^p }\)
  is the restriction of the (flat) projection
  \(\apl{\sqcup_{\adm}\Cl^p\times B }{\sqcup_{\adm}\Cl^p }\)
  to \(\sqcup_{\adm}I^p\).

  Given \(b\in V\) consider the fibre
  \(J_b\) of the (flat) morphism \(\apl{J }{V }\) (Proposition~\ref{flat proj over V}).
  By the Cartesian diagram~\eqref{eq:diagram-J-I} the scheme \(J_b\) is the pull
  back by \((\apl{\left(\sqcup_{\adm}I^p\right)_{F_0(b)}}{F_0(b)}\) of
  \(\apl{b}{F_0(b)}\). Hence \(J_b\) is a field extension of
  \(\left(\sqcup_{\adm}I^p\right)_{F_0(b)}\), in particular their Hilbert
  polynomial is
  equal~\citep
  {stacks-project}. 
  So, by~\citep[III.9.9]{hartshorne-algebraic-1977} and that \(F_0\) is
  exhaustive, the Hilbert polynomial of \((\sqcup_{\adm}I^p)_{e}\) 
  is locally constant on \(e\in \sqcup_{\adm}\Cl^p\)
  and again by \citep[III 9.9]{hartshorne-algebraic-1977} the morphism
  \(\apl{\sqcup_{\adm}I^p }{\sqcup_{\adm}\Cl^p }\) is flat. Furthermore, for all
  \(e\in \sqcup_{\adm}\Cl^p\),  \((\sqcup_{\adm}I^p)_{e}\subseteq B^{e}\)
  is an effective Cartier divisor, so by
  \citep
  {stacks-project},
  \(\sqcup_{\adm}I^p\subseteq \sqcup_{\adm}\Cl^p\times B\)
  is a relative effective Cartier divisor on
  \((\sqcup_{\adm}\Cl^p\times B)/\sqcup_{\adm}\Cl^p\).

  The scheme \(\sqcup_{\adm}I^p\)
  is the pre-image of the centre of blow up \(\bl^r_{r+1}\)
  by the morphism \(\rho\)
  restricted to \(\sqcup_{\adm}\Cl^p\times B\).
  By the universal properties of the blow up \(\bl_{r+1}^r\)
  and of the \(\Cl_r\)-Ussf
  \((\Cl_{r+1},\gamma_{r+1},p^r)\),
  there is a unique morphism \(\apl{G\!:\! \sqcup_{\adm}\Cl^p}{\Cl_{r+1}}\)
  which makes the following diagram commute.
  \[
    \begin{tikzcd}[column sep = 5em]
      \Cl_{r+1}\times B \ar{r}{\gamma_{r+1}} & \mathcal{S}^r_{r+1} \ar{d}{\bl_{r+1}^r} \\
      \sqcup_{\adm}\Cl^p \times B \ar{r}{\rho|_{\sqcup_{\adm}\Cl^p\times \Id_B}}
      \ar{u}{G\times \Id_B} &
      \mathcal{S}_r^r \\
    \end{tikzcd}
  \]

  By Lemma~\ref{lem:diagrama commutatiu gamma} and the universal property
  of the \(\Cl_r\)-Ussf
  \(\left( \Cl_r^2,\rho,q1 \right)\)
  clearly \(F\circ G=\Id_{\sqcup_{\adm}\Cl^p}\).
  We have \(G(\sqcup_{\adm}\Cl^p)\subseteq F^{-1}(\sqcup_{\adm}\Cl^p)\)
  and since \(\Delta_r(\Cl_r)\)
  is not a Type I stratum \( F^{-1}(\sqcup_{\adm}\Cl^p)\subseteq V\).
  We call \(\apl{G_0\!:\! \sqcup_{\adm}\Cl^p}{V}\)
  the corestriction of \(G\) to \(V\).

  Finally, by the universal properties of the \(\Cl_r\)-Ussf
  \((\Cl_{r+1},\gamma_{r+1},p^{r+1})\)
  and of the blow up \(\bl^r_{r+1}\)
  and by above commutative diagram and Lemma~\ref{lem:diagrama commutatiu
    gamma} the equality \(G_0\circ F_0=\Id_V\) holds.
\end{proof}

\begin{remark}\label{rmk:eq-union-typeI-stratum-Cladm}
  \(\Cl^{\adm}_r(K)=\bigsqcup_{\adm}\Cl^p(K)\).
\end{remark}

The following theorem is an immediate consequence of the previous
Theorem~\ref{theorem of birrationaliti}.
\begin{theorem}\label{the end theorem}
  Each irreducible component \(Z\) of \(\Cl_{r+1}\) is either
  \begin{enumerate}
  \item composed entirely of clusters whose \((r+1)\)'th
    section is infinitely near to the \(r\)'th (and \(F(Z)\subset \Delta_r(\Cl_r)\)),

  \item isomorphic to a type I stratum whose closure does not intersect
    \(\Delta_r(\Cl_r)\),

  \item birational to a irreducible component of a type I stratum whose closure
    intersects \(\Delta_r(\Cl_r)\), in this case \(F|_Z\)
    is a blowup map whose centre fails to be Cartier only on \(\Delta_r(\Cl_r)\).
  \end{enumerate}
\end{theorem}

\section{Examples}\label{examples}
In this section we collect a few simple examples on surfaces in which the
Universal family of relative clusters behave differently to Kleiman's iterated
blow ups in distinct ways.

We consider families \(\apl{\pi \!:\!  \mathcal{S} }{B }\)
with \(B\)
 and \(\mathcal{S}\)
projective; by Theorem~\ref{thr:existence of Cl_r} the \(r\)-Urcf
\((\Cl_r,\alpha^r_{\bullet})\)
of \(\pi\)
exists for all \(r\ge 0\).
For all the examples we fix the notation of the Section~\ref{family blowups} for
\(r=0\).
Finally, if \((X,\psi)\)
is the Usf of \(\pi\),
by Remark~\ref{equivalence sf - rcf} the scheme \(\Cl_0\)
is isomorphic to \(X\).
We refer to the elements of \(\Cl_0\) as sections and \(\Cl_1\) as clusters.


\begin{example}
  We will show a family in which there are infinitely many components of
  \(\Cl_1\) filled up with infinitely near sections which are not limits of
  distinct sections, a phenomena which does not occur on an absolute
  setting \(\apl{\mathcal{S} }{ \s(K)}\).\\

  Consider a smooth family \(\apl{\pi\!:\! \mathcal{S}}{C}\) of relative
  dimension \(2\)
  with \(C\)
  a smooth curve. There are no restrictions to the pairs of admissible sections
  because the dimension of the base is 1. Given the stratification
  \(\apl{\bigsqcup_p \Cl^p}{\Cl_0^2}\)
  of the Theorem~\ref{theorem of birrationaliti}, \(\Delta_0(\Cl_0)\)
  is the unique type II stratum and, by
  Remark~\ref{rmk:eq-union-typeI-stratum-Cladm},
  \[\left(\bigsqcup_{p}\Cl^p\setminus \Delta_0(\Cl_0)\right)\cong V\subseteq \Cl_1.\]


  There are finitely or countably many irreducible components \(C_d\)
  of \(\Cl_0\)
  with \(d\in\N\).
  For any pair of integers \(d,d'\ge 0\)
  there is a positive integer which bounds the degree of the 0-cycle intersection
  of \(\underline{\sigma}(B)\)
  and \(\underline{\sigma'}(B)\)
  for any pair of sections
  \((\sigma,\sigma')\in \left(C_d\times C_{d'}\right)(K)\).
  Given an integer \(i\ge 0\),
  we call \(D_i\)
  the locally closed subscheme of \(\Cl_0^2\)
  consisting of the pairs of sections which intersect at a 0-cycle of degree
  \(i\). The flattening stratification of \(\apl{I}{\Cl_0^2}\) is:
  \[
    \Cl_0^2 = \bigsqcup\limits_{i\ge 0} D_i\sqcup \Delta_0(\Cl_0),
  \]

  For each \(i\ge 0\)
  and each pair of positive integers \(d,d'\ge 0\)
  \(D_{i,d,d'}:=D_i\cap (C_d\times C_{d'})\)
  is either empty or an irreducible component of the stratum \(D_i\).
  We call \(\tilde{D}_{i,d,d'}\) the component
  of \(V\) isomorphic to \(D_{i,d,d'}\).\\




  Now assume for simplicity \(C=\PP^1\).
  For each \(\sigma\in \Cl_0(K)\)
  the exceptional divisor of the fibre \((\mathcal{S}_1^0)_\sigma\)
  is a rational surface isomorphic to
  \(\PP(\mathcal{N}_{\underline{\sigma}(\PP^1)/\mathcal{S}})\cong\PP(\mathcal{O}_{\PP^1}(a)\oplus
  \mathcal{O}_{\PP^1}(b))\),
  for some \(a,b\)
  in \(\Z\), which is the Hirzebruch surface, \(F_{|b-a|}\). If \(e:=b-a\ge0\),
  \[\Pic(E) = \Z[C] + \Z[F]\]
  with \(C^2=-e\),
  \(F^2=0\)
  and \(CF = 1\).
  So, an irreducible curve linearly equivalent to \(D = nC+mF\)
  intersects each fibre \(\left((\mathcal{S}_1^0)_\sigma\right)_t\)
  at exactly one point, for each \(t\in\PP^1\),
  if and only if \(1=D\cdot F=n\).
  The curve \(D\)
  is the image of a section if and only if it is irreducible and \(D=C+mF\)
  with \(m = 0\)
  or \(m\ge e\).
  The Usf \(X_E\)
  of \(\apl{E}{\PP^1}\)
  has a component \((X_E)_m\)
  for each \(m\ge e\).
Since \(\Cl_0^2\) is locally Noetherian there are a finite number of \(\tilde{D}_{i,d,d'}\)
  such that \((\sigma,\sigma)\in F(\tilde{D}_{i,d,d'})=D_{i,d,d'}\).
  So, there are infinitely many \((X_E)_m\)
  with \(F((X_E)_m)\subseteq \Delta_0(\Cl_0)\)
  filled up with sections of \(E\)
  that are not a limit of sections of some \(D_{i,d,d'}\).

  In each irreducible component of \(\Cl_0\)
  there is a locally closed subscheme \(\mathcal{N}(a,b)\),
  for each possible pair of integers \((a,b)\),
  formed by the sections \(\sigma\)
  with
  \(\mathcal{N}_{\underline{\sigma}(\PP^1)/\mathcal{S}}\cong
  \mathcal{O}_{\PP^1}(a)\oplus \mathcal{O}_{\PP^1}(b)\).
  Over \(\mathcal{N}(a,b)\),
  for each \(m>>0\),
  there is a family of rational quasiprojective varieties parametrizing all
  sections infinitely near to sections in \(\mathcal{N}(a,b)\)
  and of class \(C+mF\).
  These families form irreducible components of \(\Cl_1\).
\end{example}

\begin{example}
  We will see that \(\Cl_1\) can be smaller than \(\Cl_0\).  The phenomenon is
  due to the admissibility restriction on pairs, which does not exist on an
  absolute
  setting or when the base is a curve.\\

  Consider \(\apl{\pi\!:\! \PP^2\times \PP^2}{\PP^2}\) the projection over the
  second factor.  The scheme \(\Cl_0\) is an union of irreducible connected
  components \(C_d\) each one isomorphic to an open subscheme of
  \(\PP(K[u,v,w]^3_d)\), for \(d\in\N\) the degree of the sections.

  The following easy lemma is left to the reader.
  \begin{lemma}
    Given \(c=[P_1:P_2:P_3]\) and \(d=[Q_1:Q_2:Q_3]\) two morphisms of
    \(\apl{\PP^2}{\PP^2}\), with \(c\) non constant and
    \(Z:=\underline{c}(\PP^2)\cap\underline{d}(\PP^2)\) not equal to
    \(\underline{c}(\PP^2)\).  Then, \(Z\) is not a Cartier divisor in
    \(\underline{c}(\PP^2)\).
  \end{lemma}

\begin{corollary}
  The only pairs of admissible sections of \(\pi\) are the constant ones.
\end{corollary}

If \(e\in \Cl_1\) and \(F(e) = (c,d)\in \Cl_0^2\) with \(c\) a constant section,
\(\Bl_{\underline{c}(\PP^2)}(\PP^2\times \PP^2)\cong \Bl_q(\PP^2)\times \PP^2\),
for some \(q\in\PP^2\), and the exceptional divisor is isomorphic to
\(\apl{\PP^1\times \PP^2}{\PP^2}\), which admits constant sections only. The
cluster \(e\) must be constant, even it is infinitely near to \(c\) (e.i. if
\(c=d\)).  It is also possible to see that there are no other pairs of infinitely
near sections, so
\[\Cl_1=\Bl_{\Delta_{\PP^2}} (\PP^2\times \PP^2).\]
\end{example}

\begin{example}
  Let us now illustrate, with an explicit computation in a particular case of
  example I, that the centre ideal of a blow up (\(F|_{Z}\) in the notation of
  Theorem~\ref{the end theorem}) need not be the ideal sheaf of
  \(\Delta_r(\Cl_r)\cap F(Z)\).  Note also that in this case the centre ideal is
  not uniquely determined, although the blow up map is. This can
  happen because of the singularities in \(F(Z)\).\\

  Given a line \(R\) in \(\PP^3\), choose coordinates \([x:y:z:t]\) such that
  \(R=V(x,y)\).  We are interested in the family given by the pencil of planes
  \(\{\alpha y-\beta x =0\}\), where \([\alpha:\beta]\in \PP^1\) and whose total
  space is \(\mathcal{S}:=\Bl_R(\PP^3)\).

  Any line in \(\PP^3\) which does not meet \(R\) determines a section of
  \(\pi\).
  There is an open subscheme \(U_1\) of \(\Cl_0\) isomorphic to \(\mathbb{A}^4\)
  in such a way that each point \((a,b,c,d)\in\mathbb{A}^4\) is mapped to the
  line
  \[
    L_{(a,b,c,d)}:=V\left(
      \begin{array}{c}
        ax+by+z \\
        cx+dy+t \\
      \end{array}\right) \subseteq \mathcal{S}.
  \]
  In this example we consider all the objects of Section~\ref{family blowups}
  restricted over \(U_1\) instead of the whole \(\Cl_0\).  The morphism
  \(\apl{\rho\!:\! U_1\times U_1\times B }{U_1\times \mathcal{S}}\) is
  \[
    \rho\left((a,b,c,d),(a',b',c',d'),[u:v]\right) =
  \]
  \[
    = \left((a,b,c,d),[u:v:-a'u-b'v:-c'u-d'v],[u:v]\right).
  \]
  Let \(Y\) be \(V((a-a')(d-d')-(b-b')(c-c'))\subset U_1\times U_1\).  Given
  \(p\not=p'\in U_1\), the line \(L_p\) meets \(L_{p'}\) if and only if
  \((p,p')\in Y\) and in this case they meet always at a simple point. The
  flattening stratification of the morphism \(\apl{I}{U_1\times U_1}\) is
  \[U_1\times U_1=\Delta_0(\Cl_0)\sqcup \left(Y\setminus
      \Delta_0(\Cl_0)\right)\sqcup Y^c\]

  We now focus on the irreducible component \(Z\) of \(\Cl_1\) dominating the
  stratum \(Y\setminus\Delta_0(\Cl_0)\).  In fact, \(F(Z)=Y\) is a singular
  quadric in \(\mathbb{A}^8\), with singular locus \(\Delta_0(\Cl_0)\).

  If \([\nu:\mu]\) are the coordinates of a \(\PP^1\), the blow up
  \(\mathcal{S}^0_1\) of \(\mathcal{S}^0_0\) at \(V(ax+by+z,cx+dy+t)\), the image
  of \(\alpha_0^0\), is
  \[V\left(\mu(ax+by+z)-\nu(cx+dy+t)\right)\subseteq U_1\times \mathcal{S}\times
    \PP^1.\]

  The morphism \(\rho\) restricted to the stratum
  \((Y\setminus\Delta_0(\Cl_0))\times \PP^1\) extends to \(\mathcal{S}_1^0\).
  Over the coordinates \([\nu:\mu]\), it is \([a-a':b-b']\) or
  \([c-c':d-d']\). If both expressions are well defined, they are equal since
  \((a,b,c,d,a',b',c',d')\) belongs to \(Z\).

  If \([\omega:\eta]\) are coordinates of a \(\PP^1\), the blow up of
  \(U_1\times U_1\) at the ideal \((a-a',b-b')\) is
  \[
    V(\eta(a-a')-\omega(b-b'))\subseteq U_1\times U_1 \times \PP^1.
  \]
  The strict transform \(\tilde{Y}\) of \(Y\) under this blow up is a small
  resolution of \(Y\).  The morphism \(\rho\) extends over all
  \(\apl{\tilde{Y}\times \PP^1}{\mathcal{S}^0_1}\), over the coordinates
  \([\nu:\mu]\) it is just \([\omega:\eta]\).  That implies \(\tilde{Y}\cong Z\),
  because any two distinct points of \(\tilde{Y}\) give two distinct sections of
  \(\mathcal{S}^0_1\).  Observe that the ideal is not unique: the ideal
  \((c-c',d-d')\)
  works too.
\end{example}

\bibliographystyle{apalike}

\bibliography{bib2,grothendieck}

\end{document}